\documentclass[10pt]{article}
\usepackage{amsmath,amssymb,amscd,amsthm,graphicx,enumerate,enumitem}
\usepackage[a4paper]{geometry}

\title{Isoperimetric Properties of the Mean Curvature Flow}
\author{Or Hershkovits}
\date{}

\usepackage{amsfonts}

\numberwithin{equation}{section}
  \theoremstyle{plain}
 \newtheorem{theorem}[equation]{Theorem}
 \newtheorem*{thma}{Theorem A}
 \newtheorem*{thmb}{Theorem B}
 \newtheorem*{thmc}{Theorem C}
 \newtheorem*{thmd}{Theorem D}

 \newtheorem{lemma}[equation]{Lemma}
 
 \theoremstyle{remark}
 \newtheorem{remark}[equation]{Remark}
 \theoremstyle{remark}
\theoremstyle{definition}
 \newtheorem{definition}[equation]{Definition}

 \newtheorem{example}[equation]{Example}
\newcommand*{\rom}[1]{\expandafter\@slowromancap\romannumeral #1@}
\newenvironment{proofsketch}{\par{\noindent \it Sketch.\;\;}}{\qed\par}

\newcommand{\de}{\varepsilon}

\begin{document}
\maketitle

\begin{abstract}
In this paper we discuss a simple relation, which was previously missed, between the high co-dimensional isoperimetric problem of finding a filling with small volume to a given cycle, and extinction estimates for singular, high co-dimensional, mean curvature flow. The utility of this viewpoint is first exemplified by two results which, once casted in the light of this relation, are almost self evident. The first is a genuine, 5-lines proof, for the isoperimetric inequality for $k$-cycles in $\mathbb{R}^n$, with a constant differing from the optimal constant by a factor of only $\sqrt{k}$, as opposed to a factor of $k^k$ produced by all of the other soft methods (see \cite{MS,Gro}). The second is a 3-lines proof of a  lower bound for extinction for arbitrary co-dimensional, singular, mean curvature flows starting from cycles, generalizing the main result of \cite{GY}. We then turn to use the above mentioned relation to prove a bound on the parabolic Hausdorff measure of the space time track of high co-dimensional, singular, mean curvature flow starting from a cycle, in terms of the mass of that cycle. This bound is also reminiscent of a Michael-Simon Isoperimetric inequality. To prove it, we are lead to study the geometric measure theory of Euclidean rectifiable sets in parabolic space, and prove a co-area formula in that setting. This formula, the proof of which occupies most this paper, may be of independent interest.      
\end{abstract}

\section{Introduction}

A family of embeddings of a $k$-manifold $N$ in $\mathbb{R}^{n}$,  $\phi_{t}:N\rightarrow \mathbb{R}^{n}$ is said to evolve by mean curvature if at every point and time it satisfies the equation $\frac{d\phi}{dt}(x,t)=\vec{H}$  where $\vec{H}$ is the mean curvature vector. Starting with a smooth, compact embedding, this flow exists (smoothly) for some finite time, at which the flow encounters a singularity. As the embeddings need not disappear altogether when a singularity occurs, notions of weak solutions were desirable. Two such notions are the varifold flow of Brakke (see \cite{Bra}) and the level-set flow of Evans-Spruck (see \cite{ES1},\cite{ES2},\cite{ES3},\cite{ES4}) and Chen-Giga-Goto (see \cite{CGG}). If the initial surface has no boundary, via elliptic regularization, Ilmanen developed a notion of enhanced motion, which unites information from both approaches. 

Let $\mathbb{R}^{1,n}=\mathbb{R} \times \mathbb{R}^{n}$, $\mathbb{R}^{1,n}_{+}=\mathbb{R}_+ \times \mathbb{R}^{n}$. Given an integral $k$-cycle of finite mass in $\mathbb{R}^n$,   $T_{0}\in I_{k}^{loc}(\{0\} \times \mathbb{R}^n)$ Ilmanen's idea (motivated by the level-set approach) (see \cite{Ilm}) was to approximate the mean curvature flow (MCF) starting from $T_{0}$ by a family of ($k+1$)-dimensional translating solutions in $\mathbb{R}^n_+ $ which become more and more ``cylindrical'' and tall (in the first co-ordinate). At the limit, their horizontal sections will become a  $k$-dimensional Brakke flow in $\mathbb{R}^n$ (see Definition \ref{Brakke_Flow}) starting from $T_{0}$, while the limit of the ``down-scalings'' of these solutions will yield a current  $T \in I_{k+1}^{loc}(\mathbb{R}^{1,n}_+)$ with $\partial T = T_{0}$ that provides a measure-theoretic subsolution to the Brakke flow.

\begin{theorem}[{\cite[8.1]{Ilm}}]\label{enhanced_motion}
Let  $T_{0}\in I_{k}^{loc}(\{0\}\times \mathbb{R}^n)$ be a cycle of finite mass and of compact support. There exists a tuple $(T,\{\nu_t\}_{t\geq 0})$ called the \textbf{enhanced motion} where $T \in I_{k+1}^{loc}(\mathbb{R}^{1,n}_+) $ with $\partial T=T_0$ and  $\{\nu_{t}\}_{t\geq 0}$  is a Brakke flow with $\nu_0=\nu_{T_0}$   such that 
\begin{equation}\label{sub_solution}
\nu_{t}\geq\nu_{T_{t}},
\end{equation}
\begin{equation}\label{main_estimate}
 \underset{=}{\mathbf{M}}[(\pi_{x})_{\#}(T_{B})] \leq |B|^{1/2} \underset{=}{\mathbf{M}}[T_{0}], 
\end{equation}
and
\begin{equation}\label{second_estimate}
 \underset{=}{\mathbf{M}}[T_{B}] \leq (|B|+|B|^{1/2}) \underset{=}{\mathbf{M}}[T_{0}], 
\end{equation}
where $\pi_x:\mathbb{R}^{1,n}_+ \rightarrow \mathbb{R}^n$ is the projection to the $\mathbb{R}^n$ component and $T_B=T\lfloor B \times  \mathbb{R}^n$ for $B \subseteq \mathbb{R}$. $T$ is called the \textbf {undercurrent} and $T_t$, the $t$ time slice of $T$, is called the \textbf{underflow}.
\end{theorem} 

\begin{remark}
Theorem \ref{enhanced_motion} is stated and proved in \cite{Ilm} for arbitrary ambient Riemannian manifolds. 
\end{remark}

The main conceptual observation of this paper is a relation, that was previously missed, between the above theorem, coupled with an extinction estimate, to the isoperimetric problem and isoperimetric inequalities.   Let us briefly discuss those.

In the fundamental paper \cite{FF}, Federer and Fleming proved an isoperimetric inequality (now bearing their names) stating that there exists a constant $c=c(k,n)$, depending both on the dimension and the co-dimension,  such that for every integral cycle $N \in I_{k}(\mathbb{R}^{n})$ with finite mass and compact support, there exists some $ F \in I_{k+1}(\mathbb{R}^{n})$ such that
\begin{equation}\label{isop}
\underset{=}{\mathbf{M}}[F] \leq c  \underset{=}{\mathbf{M}}[N]^\frac{k+1}{k} 
\end{equation}
and $\partial F = N$ (see also \cite[29,30]{Sim}). In particular, by \eqref{isop}, by a compactness result proved  in \cite{FF},  since mass is lower semi-continuous under weak convergence and since the boundary operator is continuous, it follows from the direct method of the calculus of variations that  $N$ has an \textbf{optimal filling} $ F_{\textrm{opt}} \in I_{k+1}(\mathbb{R}^{n})$ which has the property that $\partial F_{\textrm{opt}}=N$ and $\underset{=}{\mathbf{M}}( F_{\textrm{opt}}) \leq \underset{=}{\mathbf{M}}( F) $  for every $ F \in I_{k+1}(\mathbb{R}^{n})$ with $\partial F=N$.  

Thirteen years later, Michael and Simon showed that the constant $c$ can be taken to be independent of the co-dimension (see \cite{MS}). This was soon generalized to the case of non-positively curved Riemannian manifolds (see \cite{HS}). Another argument that applies to a much more general setting than Euclidean spaces was given in \cite{Gro} where one also gets the estimate $c=O(k^k)$.  Finally, in \cite{Alm} Almgren proved an ``optimal'' isoperimetric inequality in the Euclidean case: the constant $c$ in (\ref{isop})  corresponds to the case of a standard sphere enclosing a disk (i.e. $c_k=1/(k\omega_k^{1/k}) \approx  1/\sqrt{k}$) and equality is achieved if and only if $N$ is the standard sphere. 

\vspace{5 mm}

As a first instance of the above mentioned relation between Theorem \ref{enhanced_motion} and the isoperimetric problem,  we can generalize the main theorem of \cite{GY} (with a constant worse by a factor of two). Since the proof is so short, there is no need to postpone it for later sections.

\begin{thma}[Lower Bound on Extinction]\label{main_thm1}
If $T_0\in I_{k}^{loc}(\{0\}\times \mathbb{R}^n)$ is of finite mass and compact support, $(T,\{\nu_{t}\}_{t\geq 0})$ is an enhanced motion corresponding to $T_0$, and $F_{\textrm{opt}}$ is the optimal filling of $T_0$, then the extinction time of $\{\nu_t\}_{t\geq 0}$, $\tau=\sup\{t\geq 0 \;|\;\nu_{t}(\mathbb{R}^n)>0\}$, satisfies
\begin{equation}
\tau \geq \left(\underset{=}{\mathbf{M}}[F_{\textrm{opt}}] /\underset{=}{\mathbf{M}}[T_0]\right)^2.
\end{equation} 
\end{thma}
\begin{proof} If the extinction time was  $\tau<\tau_0=\left(\underset{=}{\mathbf{M}}[F_{\textrm{opt}}] /\underset{=}{\mathbf{M}}[T_0]\right)^2$, by (\ref{sub_solution}) $\partial T_{[0,(\tau+\tau_0)/2)}=T_0$, and setting $S=(\pi_x)_{\#}(T_{[0,(\tau+\tau_0)/2)})$ we also have $\partial S=T_0$. By (\ref{main_estimate}) this implies
\begin{equation}
\underset{=}{\mathbf{M}}[S] < \underset{=}{\mathbf{M}}[T_0]\tau_0^{1/2} = \underset{=}{\mathbf{M}}[F_{\textrm{opt}}] 
\end{equation}
which would provide a filling of $T_0$ that is better than the optimal one.  
\end{proof}

\begin{remark}
The co-dimension one case of Theorem A implies, in particular, that if $D\subseteq \mathbb{R}^n$ is a bounded set with smooth boundary, then the extinction time for the \textit{level-set flow} of $\partial D$ satisfies  $ \tau \geq \left(\mathcal{H}^n(D)/\mathcal{H}^{n-1}(\partial D)\right)^2$, where $\mathcal{H}^k$ denotes the standard $k$-dimensional Hausdorff measure. This is true since any co-dimension one Brakke flow remains supported in the level set flow. The main result of \cite{GY} states that in such a case $ \tau \geq 2\left(\mathcal{H}^n(D)/\mathcal{H}^{n-1}(\partial D)\right)^2$.  Theorem A therefore loses a factor of two, but it allows for the generalization of \cite{GY}, both from level set flows to (special) Brakke flows, and from co-dimension one to arbitrary co-dimensions.  
\end{remark}

\begin{remark}
A partial explanation for why lower bounds on extinctions are interesting is their evident relation to the far harder and more central question of mass drop for Brakke flows (see \cite{Ilm,MSC}).
\end{remark}

\vspace{5 mm}

Next, we show how the above relationship can be used to derive isoperimetric inequalities with good constants. 
\begin{thmb}[MCF Spatial Isoperimetric Inequality]\label{main_thm2}
Let $(T,\{\nu_{t}\}_{t\geq 0})$  be an  enhanced motion  starting from an integral $k$-cycle $T_0$ of finite mass with compact support in $\mathbb{R}^{n}$. Then $\{\nu_{t}\}$ becomes extinct at finite time $\tau$, and taking $S=(\pi_{x})_{\#}(T_{[0,\tau]}))$, we have $\partial S=T_{0}$ and 
\begin{equation}
\underset{=}{\mathbf{M}}[S]\leq \frac{1}{\sqrt{4\pi}}  \underset{=}{\mathbf{M}}[T_{0}]^\frac{k+1}{k}. 
\end{equation}
\end{thmb}

As with Theorem A, the proof of Theorem B is very short, and will therefore again be included in the introduction. We will need the following lemma, which is proved by a standard use (at least in the smooth case) of the monotonicity formula (c.f \cite{Hui,Ilm2} and Lemma  \ref{mon_for}), and which will be proved in Section 2.    
\begin{lemma}\label{extinct_bd}
Let $\{\nu_{t}\}_{t\geq 0}$ be an integral Brakke flow (see Definition \ref{Brakke_Flow}) with $\nu_{0}(\mathbb{R}^n)<\infty$. Then the flow becomes extinct in finite time $\tau$ with
\begin{equation}
\tau \leq \frac{\nu_{0}(\mathbb{R}^n)^{\frac{2}{k}}}{4\pi}. 
\end{equation}
\end{lemma}

\begin{proof}[Proof of Theorem B] Taking $\tau=\nu_{0}(\mathbb{R}^n)^{\frac{2}{k}}/(4\pi)$ , Lemma \ref{extinct_bd} gives the first part of the theorem.  We further see, as in the proof of Theorem A, that for every $\de>0$, $S_\de:=(\pi_x)_{\#}(T_{[0,\tau+\de)})$     satisfies $\partial S=T_0$ and $\underset{=}{\mathbf{M}}[S_\de]\leq \underset{=}{\mathbf{M}}[T_{0}](\tau+\de)^{1/2}$. By \eqref{main_estimate} $S_\de \rightarrow S$ in the mass norm, and therefore also weakly. By the continuity of boundary and by the definition of $\tau$, we obtain:
\begin{equation}
\underset{=}{\mathbf{M}}[S] \leq \frac{1}{\sqrt{4\pi}} \underset{=}{\mathbf{M}}[T_{0}]^{\frac{k+1}{k}}.
\end{equation}
\end{proof}

In light of the above, Theorem B shows that the spatial track  of a singular MCF provides a filling that satisfies such a Gromov-Michael-Simon isoperimetric inequality, which is worse than the optimal one by a factor of only $\sqrt{k}$. Theorem B is particularly interesting  when the underflow is in fact a Brakke flow,  as then the filling is local.

\begin{remark}
The proof of Theorem B can be considered as a genuine proof of the  isoperimetric inequality with a relatively good constant (like all  the other proofs of the isoperimetric inequality that were mentioned, expect for \cite{FF}, it relies implicitly on the results from \cite{FF}). While not giving the optimal constant of \cite{Alm}, our proof is only a few lines long, as opposed to a very long paper. To the best of the author's knowledge, other than the proof in \cite{Alm}, our argument yields a better constant than all other proofs.   
\end{remark}

\begin{remark}
The methods described here provide a new approach for studying the isoperimetric problem in arbitrary co-dimension on general Riemannian  manifolds, which amounts to estimating the extinction time of mean curvature flow starting from the initial cycle.    
\end{remark}

Prior uses of the mean curvature flow in studying the isoperimetric problem include  \cite{Sch}, where the level-set power MCF on mean convex hypersurfaces was used to derive the optimal isoperimetric inequality in co-dimension one, and a Euclidean isoperimetric inequality for surfaces in simply connected 3-manifolds with non-positive sectional curvature (which was proven originally in \cite{Kle}),  and \cite{Top} where the curve shortening flow was used to obtain an optimal isoperimetric inequality on surfaces. In both cases the argument is based on a monotonicity of certain surface-area volume functions, and is quite different from the one here.

\vspace{5 mm}

In order to state the next result, we  first recall the notions of parabolic metric and parabolic Hausdorff measure. If $A\in \mathbb{R}^{1,n}$  is a space-time track of a mean curvature flow, then given $\lambda>0$, $\lambda A$ is not such a track, but  $\{(\lambda^2 t,\lambda x) \;\;|\;\;(t,x)\in A\}$ is. The same is true (measure theoretically and in the sense of the underlying rectifiable sets) for the space-time track of a Brakke flow. To study scale-invariant properties we would therefore need a metric that respects those scalings.
 
\begin{definition}[\cite{Whi}]
The \textbf{parabolic metric} on $\mathbb{R}^{1,n}$ is defined to be 
\begin{equation}
d_{par}((t,x),(s,y))=\mathrm{max}\{\sqrt{|t-s|},|x-y|\}.
\end{equation}
The Hausdorff measure corresponding to $d_{par}$ will be called the \textbf{parabolic Hausdorff measure} and will be denoted by $\mathcal{H}^{*}_{par}$.
\end{definition}

\begin{remark} Note that a $k+1$ plane has parabolic Hausdorff dimension $k+1$ if it is perpendicular to $\partial_t$ and parabolic Hausdorff dimension $k+2$ if it has some $\partial_t$ component.  As mean curvature flow is  a flow in time, it is therefore reasonable to measure the ($k+2$)-Hausdorff measure of the space time track $\mathcal{H}^{k+2}_{par}$.
\end{remark}
\begin{remark}
The parabolic Hausdorff measure was first introduced to the study of mean curvature flow by White for his dimension reduction principle (see \cite{Whi}). In there, the question concerned the parabolic Hausdorff dimension of certain sets (the singular stratum) and the only property of the measure that was used  was the above mentioned scaling. The relationship between the total measure of a set and its time slices was used by Federer's general co-area inequality (see \cite[2.10.25]{Fed}). It will be one of the main technical objectives of this current paper to relate the horizontal measures of the slices of a Euclidean rectifiable set in space-time to the parabolic measure of the entire set in a more precise way (i.e. to obtain a co-area formula type result). 
\end{remark}

\begin{thmc}[Parabolic Measure Estimate for the Space-Time Track of a MCF]\label{main_thm3}
There exists some universal constant $C=C(k)$ with the following property: Let $(T,\{\nu_{t}\}_{t\geq 0})$ be an enhanced motion starting from an  integral $k$-cycle $T_0$ of finite mass and compact support in $\mathbb{R}^{n}$. Letting $T=\tau(X,\theta,\xi)$ (i.e. $T$ is the integral current corresponding to the rectifiable set $X$, the multiplicity $\theta$  and the orientation $\xi$), set  $\mu$ to be the rectifiable parabolic radon measure corresponding to $(X,\theta)$, i.e. 
\begin{equation}
\mu=\theta\mathcal{H}^{k+2}_{par}\lfloor X.
\end{equation} 
Then $\mu$ is well defined (see remark below) and
\begin{equation}
\mu(\mathbb{R}^{1,n}_+) \leq C(k) \underset{=}{\mathbf{M}}[T_0]^{\frac{k+2}{k}}. 
\end{equation}  
\end{thmc}      

Several remarks are in order.

\begin{remark}
Rectifiable sets are defined as a union of Lipschitz images and a set of $\mathcal{H}^{k+1}$ measure zero, which is inconsequential from the point of view of a current supported on the set. Thus, in order for the above theorem to make sense, one needs to show that for a set $B\subseteq \mathbb{R}^{1,n}$ we have $\mathcal{H}^{k+1}(B)=0$  implies $\mathcal{H}^{k+2}_{par}(B)=0$. This is part of the content of Lemma \ref{absolute_continuity}.
\end{remark}

\begin{remark}
In order for a scale invariant parabolic isoperimetric inequality concerning the undercurrent to be meaningful, one must check that the undercurrent construction itself is scale invariant. This is done in Lemma \ref{scaling_under}.
\end{remark}

\begin{remark}
Note that the constant $C$ depends only on the dimension of the current, and not on the dimension of the ambient space. Thus, the above estimate is reminiscent of the Michael-Simon isoperimetric inequality, as discussed above.
\end{remark}

To prove Theorem C, we are led to study the geometric measure theory of \underline{Euclidean}  rectifiable currents in parabolic space and the relationship between the Euclidean Hausdorff  measure of time slices of such sets and the total parabolic Hausdorff measure. As it turns out, the co-area formula in such a situation takes the form of Fubini's theorem without any co-area factor. More precisely, we will have the following theorem, which is perhaps of some interest in its own right.

\begin{thmd}[Parabolic Co-Area]\label{strong_fubini}
Let $\mathcal{M} \subseteq \mathbb{R}^{1,n}$ be a Euclidean ($k+1$)-rectifiable set of finite ($k+1$) dimensional Hausdorff measure  and let $g:\mathcal{M}\rightarrow \mathbb{R}$. Then
\begin{equation}
\int_{\mathcal{M}} gd\mathcal{H}^{k+2}_{par}=c_1(k)\int_{\mathbb{R}^{1,0}}\left(\int_{\mathcal{M}_{t}}gd\mathcal{H}^{k}\right)d\mathcal{H}^{2}_{par}(t)
\end{equation}   
where $c_1(k)$ is some universal constant.
\end{thmd}

\begin{remark}
This should be compared with the Euclidean situation where the co-area formula takes the form
\begin{equation}
\int_{\mathcal{M}} g|\nabla^\mathcal{M} t|d\mathcal{H}^{k+1}=\int_{\mathbb{R}}\left(\int_{\mathcal{M}_{t}}gd\mathcal{H}^{k}\right)d\mathcal{H}^{1}(t).
\end{equation}
\end{remark}
\begin{remark}
The absence of a co-area factor is not too surprising: considering, say, a smooth $k+1$ submanifold in $\mathbb{R}^{1,n}$ we see that if the tangent plane at some point had a time-like direction, the parabolic blow-ups at that point will contain the vector $\partial_t$, and so we are always in the ``split'' Fubini situation infinitesimally. If it were perpendicular to time, then it should not contribute to the $\mathcal{H}^{k+2}_{par}$ measure anyway. 
\end{remark}
Provided Theorem D, Theorem C follows easily.
\begin{proof}[Proof of Theorem C assuming Theorem D]
\eqref{sub_solution} and the fact that the mass always decrease along a Brakke flow, imply that for every $t$, $\underset{=}{\mathbf{M}}[T_t]\leq \nu_t (\mathbb{R}^{n}) \leq \nu_0(\mathbb{R}^{n})$. The extinction estimate of Lemma \ref{extinct_bd}, together with Theorem D give the desired result, noting that on $\mathbb{R}_+$, $\mathcal{H}^2_{par}$ is a multiple of the Lebesgue measure (see Example \ref{one_d}). 
\end{proof} 

The organization of the paper is as follows: In Section 2 we collect some preliminary results, and  in Section 3 we study the geometric measure theory (GMT) of Euclidean currents in parabolic space and in particular prove Theorem D. The rest, as the perceptive reader have noticed, was already proved in the introduction.

\vspace{5 mm}

\noindent\textbf{Acknowledgments:} I would like to thank Robert Haslhofer, Robert Kohn and Bruce Kleiner for many useful discussions. I would like to thank Jacobus Portegies for carefully reading and commenting on an earlier version of this note.

\section{Preliminaries}

We recall the notion of a Brakke flow (see \cite[3.2]{Bra}). We will follow the slightly different definition appearing in \cite{Ilm}. Let
\begin{equation}
\overline{D}_{t_{0}}f=\overline{\mathrm{lim}}_{t\rightarrow t_{0}}\frac{f(t)-f(t_{0})}{t-t_{0}}
\end{equation}
for $f:\mathbb{R}\rightarrow\mathbb{R}$.

\begin{definition}[{\cite[6.2-3]{Ilm}}]\label{Brakke_Flow}
A family of Radon measures $\{\nu_{t}\}_{t\geq 0}$ on $\mathbb{R}^n$ is a $k$-dimensional \textbf{Brakke flow} if for all $t\geq 0$ and all $\phi\in C^{1}_{c}(M,\mathbb{R}_{+})$ we have $\overline{D}_{t}\nu_{t}(\phi)\leq\mathcal{B}(\nu_{t},\phi)$ where
\begin{equation}
\mathcal{B}(\nu,\phi)=\int -\phi H^{2}+\nabla\phi \cdot S^{\perp} \cdot \vec{H} d\nu
\end{equation}
whenever  $\nu\lfloor\{\phi>0\}$ is radon $k$  rectifiable, $|\delta V|\lfloor \{\phi>0\}$ is a radon measure absolutely continuous w.r.t.  $\nu\lfloor\{\phi>0\}$ for $V=V_{\nu}\lfloor\{\phi>0\}$ and when $\phi H^{2}$ is integrable. Here $S$ is the approximate tangent space and we confuse a subspace with the projection operator to it. If either of the above conditions is not satisfied, we let $\mathcal{B}(\nu,\phi)=-\infty$ . A Brakke flow is called \textbf{integral} if for a.e. $t\geq 0$, $\nu_{t}\in IM_{k}(\mathbb{R}^n)$, the space of integer rectifiable radon measures.
\end{definition}

We will need the following generalization of Huisken's monotonicity formula (see \cite{Hui}) to the context of Brakke flows (see \cite{Ilm2}). 

\begin{theorem}\label{mon_for}
Let $\{\nu_{t}\}_{t\geq 0}$ be a $k$-dimensional integral Brakke flow in $\mathbb{R}^{n}$ with $\nu_{0}(\mathbb{R}^n)<\infty$ and let $t_{1}<t_{2}<\tau$  and $p\in\mathbb{R}^{n}$, then we have
\begin{equation}
\int \frac{1}{(4\pi(\tau-t_{2}))^{k/2}}e^{\frac{-|x-p|^{2}}{4(\tau-t_{2})}} d\nu_{t_{2}}(x) \leq \int \frac{1}{(4\pi(\tau-t_{1}))^{k/2}}e^{\frac{-|x-p|^{2}}{4(\tau-t_{1})}} d\nu_{t_{1}}(x). 
\end{equation}
$\Box$
\end{theorem} 

We can now give a proof of the extinction estimate of Lemma \ref{extinct_bd}. As stated in the introduction, this extinction estimate is standard, at least in the smooth case (see \cite[3.2.16]{Man} for instance). The generalization to Brakke flows is straight-forward, and will be given here for the sake of completeness.

\begin{proof}{Proof of Lemma \ref{extinct_bd}} Take $t>0$ at which the flow is integral (which happens a.e.) and not extinct, let $p$ be a point at which the approximate tangent space of $\nu_t$ exists  and has multiplicity $\theta_0 \geq 1$ and let $s>t$. Then by the monotonicity formula we get

\begin{equation}
\int \frac{1}{(4\pi(s-t))^{k/2}}e^{\frac{-|x-p|^{2}}{4(s-t)} }d\nu_{t}(x) \leq \int \frac{1}{(4\pi s)^{k/2}}e^{\frac{-|x-p|^{2}}{4s}} d\nu_{0}(x)\leq \frac{1}{(4\pi s)^{k/2}}\nu_{0}(\mathbb{R}^n)
\end{equation}
and taking the limit $s \rightarrow t$ we obtain
\begin{equation}
1\leq \theta_0  \leq \frac{1}{(4\pi t)^{k/2}}\nu_{0}(\mathbb{R}^n)
\end{equation} 
so $t \leq \frac{\nu_{0}(\mathbb{R}^n)^{\frac{2}{k}}}{4\pi}$  holds for a.e time prior to the extinction time and we are done. 
\end{proof}

In what follows, we indicate why Ilmanen's construction of the enhanced motion (\cite{Ilm}) respects the natural parabolic scalings on $\mathbb{R}^{1,n}$. We also include, for the reader's convenience, Ilmanen's heuristics for why $T$ should be seen as the space-time track of the mean curvature flow (see \cite[2.2]{Ilm}). For both purposes, we need to describe the construction in some more detail.    
Let $T_{0}\in I_{k}(\{0\} \times \mathbb{R}^{n})$ be an integral cycle of finite mass. For  $Q \in I_{k+1}(\mathbb{R}^{1,n})$ and $\epsilon>0$,  Ilmanen (\cite[2.1]{Ilm}) defined the functional
\begin{equation}
I^{\epsilon}[Q]=\frac{1}{\epsilon}\int e^{-z/\epsilon}d\nu_{Q},
\end{equation}
where $(z,x)\in \mathbb{R}^{1,n}$. By the direct method of the calculus of variations, he produced currents  $P^{\epsilon} \in I_{k+1}^{loc}(\mathbb{R}^{1,n})$ minimizing $I^\epsilon$ subject to the constraint of having boundary $T_{0}$, which additionally turn out to be supported on $\mathbb{R}^{1,n}_{+}$. The Euler-Lagrange equation of this functional is 
\begin{equation}\label{EL}
\vec{H}+S^{\perp}\cdot \frac{\omega}{\epsilon}=0,
\end{equation} 
where $\vec{H}$ is the generalized mean curvature vector, $\omega=(1,0,\ldots,0)$ and $S$ the approximate tangent space (with the usual abuse of notation identifying a subspace with the projection to it). $P^{\epsilon}$ is thus a translating solution for the MCF with velocity $v=-\frac{\omega}{\epsilon}$. Letting $\kappa_\epsilon(z,x)=(\epsilon z ,x)$ and $t=\epsilon z$, Ilmanen defines $T^{\epsilon}=(\kappa_{\epsilon})_{\#}(P^{\epsilon})$ and Ilmanen's undercurrent $T$ is, by definition, a sub-limit of those $T^{\epsilon}$ as $\epsilon \rightarrow 0$ .
\begin{remark}\label{heuristics}
Ilmanen sees those $T^{\epsilon}$ as an approximation for the space-time track of the mean curvature flow starting form $T_0$. The reason is the following (see \cite[2.2]{Ilm}): As it turns out, for $\epsilon<<1$ , $P^\epsilon$ are of height $\cong \frac{C}{\epsilon}$ and are almost cylindrical. Slicing $P^\epsilon$ at some $z$ we obtain
\begin{equation}
\vec{H}_{T^{\epsilon}_t} =    \vec{H}_{P^{\epsilon}_z} \approx \vec{H}_{P^{\epsilon}}.
\end{equation} 
Letting $\mbox{MCF}^{k}(M,t)$ be flow by mean curvature of a $k$ sub-manifold $M$ for time duration $t$ and $\mbox{HMCF}^{k}(M,t)$ be flow by only the horizontal part of the mean curvature, we see that for $s>t$
\begin{equation}
\begin{aligned}
&T^{\epsilon}_s=P^{\epsilon}_{s/\epsilon} = \mbox{MCF}^{k+1}(P^{\epsilon},s-t)_{t/\epsilon} \approx \mbox{HMCF}^{k+1}(P^{\epsilon},s-t)_{t/\epsilon}= \\
& \mbox{MCF}^{k}(P^{\epsilon}_{t/\epsilon},s-t)=\mbox{MCF}^{k}(T^{\epsilon}_{t},s-t). 
\end{aligned}
\end{equation} 
$T$, being a sub-limit of the $T^{\epsilon}$ is thus seen (intuitively) as the space-time track. 
\end{remark} 

Both classical mean curvature and Brakke flow are invariant under parabolic rescalings. The same is true for the undercurrent
\begin{lemma}\label{scaling_under}
Let $\lambda>0$  and let $\eta_\lambda,S_\lambda:\mathbb{R}^{1,n} \rightarrow \mathbb{R}^{1,n}$ be parabolic and Euclidean rescaling by $\lambda$, i.e $\eta_\lambda(t,x)=(\lambda^2 t,\lambda x)$ and $S_\lambda(t,x)=(\lambda t, \lambda x)$. If $T$ is an undercurrent corresponding to $T_0$ then $(\eta_\lambda)_{\#}(T)$ is an undercurrent corresponding to $(S_\lambda)_{\#}(T_0)$.
\end{lemma}
\begin{proof}   Letting $P^{\epsilon}(T_0)$ be a minimizer of $I^{\epsilon}$ with boundary $T_0$ we see that $(S_\lambda)_{\#}(P^{\epsilon})$ is a minimizer of $I^{\lambda\epsilon}$ with boundary $(S_\lambda)_{\#}(T_0)$. Thus one can take $P^{\lambda\epsilon}((S_\lambda)_{\#}(T_0))=(S_\lambda)_{\#}(P^{\epsilon}(T_0))$  and so
\begin{equation}
T^{\lambda \epsilon}((S_\lambda)_{\#}(T_0))=(\kappa_{\lambda \epsilon})_{\#}((S_\lambda)_{\#}(P^{\epsilon}(T_0)))= (\kappa_{\lambda}\circ S_\lambda )_{\#}(T^\epsilon(T_0))=(\eta_{\lambda})_{\#}(T^\epsilon(T_0)).
\end{equation}
Thus, we get the desired scaling in the level of the subsequences and so at a  (possible) limit. 
\end{proof}

\begin{remark}
Without the above lemma, a parabolic Hausdorff measure estimate regarding the undercurrent would have been rather meaningless. This is not the case.
\end{remark}

\section{Parabolic GMT of Euclidean Rectifiable Sets}
This section is divided as follows: In Section \ref{basic} we will explore some basic properties of the parabolic Hausdorff measure, In Section \ref{horizontal_set} we will show that  infinitesimally spatial Euclidean ($k+1$)-rectifiable sets (see Definition \ref{inf_spa_def}) with finite volume are negligible in the parabolic setting, in Section \ref{time_advancing} we will deal with sets that have a time-like component a.e. and in Section \ref{Fubini_conclusion} we will prove the parabolic co-area formula, Theorem  D.
 
\subsection{Basic properties and Examples}\label{basic}

Recall that we are considering the space $\mathbb{R}^{1,n}=\mathbb{R}\times \mathbb{R}^{n}$. Points $p\in \mathbb{R}^{1,n}$ will be denoted by $p=(t,x)$. The $\mathbb{R}$ factor is called the time direction and the $\mathbb{R}^{n}$ factor is called the space direction. On $\mathbb{R}^{1,n}$ we consider two metrics: the standard Euclidean one $d$ with corresponding Hausdorff measure $\mathcal{H}^{*}$ and the parabolic metric $d_{par}$ with corresponding Hausdorff measure $\mathcal{H}^{*}_{par}$. $diam$ will stand for the Euclidean diameter while $diam_{par}$ will stand for the parabolic one. Rectifiable will mean, unless otherwise stated, Euclidean-rectifiable.
\begin{remark}
By the Caratheodory criterion it is clear that $\mathcal{H}^{k}_{par}$ is Borel.
\end{remark}
 
The first thing we will see is the following.    

\begin{lemma}\label{absolute_continuity}
Let $A$ be a ($k+1$)-rectifiable set in $\mathbb{R}^{1,n}$.
\begin{enumerate}
\item If $\mathcal{H}^{k+1}(A)=0$ then $\mathcal{H}^{k+2}_{par}(A)=0$.
\item If $\mathcal{H}^{k+1}(A)<\infty$ then $\mathcal{H}^{k+2}_{par}(A)<\infty$. 
\end{enumerate}
\end{lemma}
\begin{proof}  For ($\mathit{1}$), since $\mathcal{H}^{k+1}(A)=0$, for every $\delta>0$ and $\epsilon$ there is a $\delta$ cover of $A$ by cubes $\{C_i\}$, parallel to the axes with $\sum diam(C_i)^{k+1}<\epsilon$  (by enlarging an initial small covering and swallowing the constant multiplicative factor). Looking at $C_i$ in the parabolic metric, we see that  $diam_{par}(C_i)\cong \sqrt{diam(C_i)}$.  Now, slice each $C_i$ to rectangular boxes $\{D_i^j\}$ of time-like sides of length $diam(C_i)^2$ and space-like sides of length $diam(C_i)$ . This way, the parabolic diameter of the boxes will be smaller than the Euclidean diameter of the original cube and so they provide a parabolic $\delta$ cover of $A$. We will need $[\frac{1}{diam(C_i)}]$ such boxes to cover the cube $C_i$, and
\begin{equation}
\sum_{j=1}^{[\frac{1}{diam(C_i)}]}diam_{par}(D_i^j)^{k+2}\leq[\frac{1}{diam(C_i)}] diam(C_i)^{k+2}=diam(C_i)^{k+1}
\end{equation}
and so $\mathcal{H}^{k+2}_{par,\delta}(A) < \epsilon$ and we are done. The proof of ($\mathit{2}$) is similar (see also Lemma \ref{horizontal_part}). 
\end{proof}

\begin{remark}
The first part of the lemma allows us to measure the ($k+2$)-parabolic measure of a ($k+1$)-rectifiable set in a well defined manner.
\end{remark}

The following example shows that in $\mathbb{R}^{1,0}$ there is no real difference between $\mathcal{H}^{2}_{par}$  and the standard one dimensional Lebesgue measure $\mathcal{L}^{1}$. 

\begin{example}\label{one_d}
On $\mathbb{R}^{1,0}\cong \mathbb{R}$ we have
\begin{equation}
\mathcal{H}^{2}_{par}=\frac{\alpha(2)}{2\alpha(1)}\mathcal{L}^{1}.
\end{equation} 
\end{example}
\begin{proof}
 It suffices to check it for intervals. Taking $\delta>0$ and a parabolic $\delta$ cover of $[a,b]$   $\{C_{i}\}$  we have $diam_{Euc}(C_{i})=diam_{par}(C_{i})^{2}$  and so it is a Euclidean $\delta^{2}$ cover of $[a,b]$ and
\begin{equation}
\alpha(2)\sum\left(\frac{diam_{par}(C_i)}{2}\right)^{2}=\frac{\alpha(2)}{2\alpha(1)}\alpha(1)\sum\frac{diam_{Euc}(C_{i})}{2}.
\end{equation}
\end{proof}

More generally in the full dimensional case, the parabolic Hausdorff measure is identical, up to a constant, to the Lebesgue measure.

\begin{lemma}[Top dimensional compatibility]\label{top_dim_comp}
There exist some constants $c_i=c_i(k)>0$ ($i=1,2$) such that on $\mathbb{R}^{1,k}$ we have
\begin{equation}
\mathcal{H}^{k+2}_{par} = c_2\mathcal{L}^{k+1} = c_1 \mathcal{H}^{2}_{par} \times \mathcal{H}^{k}.
\end{equation} 
\end{lemma}
\begin{proof} The  second equality is clear from Example \ref{one_d}. For the first equality,  note that by Lemma \ref{absolute_continuity}, $\mathcal{H}^{k+2}_{par}$ is a radon measure that is absolutely continuous w.r.t the Lebesgue measure. As both $\mathcal{H}^{k+2}_{par}$ and  $\mathcal{L}^{k+1}$ are invariant under translations, by Radon-Nikodym we obtain
\begin{equation}
\mathcal{H}^{k+2}_{par} = c_2 \mathcal{L}^{k+1}
\end{equation} 
for some $c_2\geq 0$. In order to conclude, it will suffice to show that $\mathcal{H}^{k+2}_{par}([0,1]\times [0,1]^{k}) > 0$. Otherwise, for every $\epsilon$, there would be a cover $\{C_i\}$ such that $\sum diam_{par}(C_{i})^{k+2} < \epsilon$, but then by perhaps enlarging $C_{i}$ a little bit, we get $\tilde{C}_{i}$ of the form $\tilde{C}_{i}=[a_i,b_i]\times D_{i}$ and with 
\begin{equation}
\sqrt{|b_i-a_i|}=diam_{par}(\tilde{C}_{i})=diam_{par}(C_{i})=diam(D_i)
\end{equation}
but then
\begin{equation}
\sum \mathcal{L}^{k+1}(\tilde{C}_i) = \sum diam_{par}(\tilde{C}_{i})^{k+2} < \epsilon  
\end{equation}
so the $C_{i}$ can not be a cover of the unit cube. 
\end{proof}

The situation with lower dimensional parabolic Hausdorff measures is very different, as Example \ref{1d_example} (and indeed the entire Section \ref{time_advancing}) will indicate.

\subsection{Infinitesimally Spatial Rectifiable Sets}\label{horizontal_set}
This subsection deals with the validity of the parabolic co-area formula (Theorem D) for rectifiable sets which are infinitesimally spatial.
\begin{definition}\label{inf_spa_def}
A $(k+1)$ rectifiable set $B$ in $\mathbb{R}^{1,n}$ is called \textbf{infinitesimally spatial} if for $\mathcal{H}^{k+1}$ a.e $p \in B $ we have $\partial_t \perp T_p B$.  
\end{definition} 

\begin{lemma}\label{horizontal_part}
If $B$ be is a ($k+1$)-infinitesimally spatial set in $\mathbb{R}^{1,n}$ and $\mathcal{H}^{k+1}(B)<\infty$ then $\mathcal{H}^{k+2}_{par}(B)=0$. 
\end{lemma}

\begin{proof}  The argument is a refined version of the one in Lemma \ref{absolute_continuity}. Fix $\beta>0$. For $l=1,2,\ldots$ let $B^{l}_{\beta}$ be the set of points at which the set it $\beta$ close to being spatial at scales  $<1/l$. More precisely, $B^{l}_{\beta}$ consists of those $(t,x)$ in $B$ such that for every $r<1/l$  we have
 \begin{equation}\label{ratio_bd1}
1-\beta \leq \frac{\mathcal{H}^{k+1}(B\cap I^{n+1}((t,x),r))}{(2r)^{k+1}}  \leq 1+\beta,
\end{equation}
and also
\begin{equation}\label{ratio_bd2}
1-2\beta \leq \frac{\mathcal{H}^{k+1}(B\cap I^{n+1}((t,x),r)\cap ([t-\beta r,t+\beta r]\times \mathbb{R}^{n}))}{(2r)^{k+1}}  \leq 1+2\beta,
\end{equation} 
where $ I^{n+1}((t,x),r)$ is the rectangle parallel to the axes with center $(t,x)$ and (Euclidean) sides $2r$. We claim that for $\beta$ sufficiently small, if $(t,x)\in B^{l}_{\beta}$ then for every $r<1/(2l)$  
\begin{equation}\label{fewer}
B^{l}_{\beta} \cap I^{n+1}((t,x),r)\cap ([t-3\beta r,t+3\beta r]\times \mathbb{R}^{n})^c =\emptyset.
\end{equation}
Otherwise, given $(s,y)$ in the left hand side of (\ref{fewer}) we get that
\begin{equation}
B\cap I^{n+1}((s,y),r)\cap ([s-\beta r,s+\beta r]\times \mathbb{R}^{n}) \subseteq B\cap I^{n+1}((t,x),2r)
\end{equation}
but
\begin{equation}
(I^{n+1}((s,y),r)\cap ([s-\beta r,s+\beta r]\times \mathbb{R}^{n})) \bigcap (I^{n+1}((t,x),2r)\cap ([t-2\beta r,t+2\beta r]\times \mathbb{R}^{n}))=\emptyset. 
\end{equation}
For small $\beta$, this would contradict (\ref{ratio_bd1}) for the point $(t,x)$ at scale $2r$, as by the above disjointness and by using (\ref{ratio_bd2}) first for $(t,x)$ at scale $2r$ and then for $(s,y)$ at scale $r$ we get
\begin{equation}
\frac{\mathcal{H}^{k+1}(B\cap I^{n+1}((t,x),2r))}{(4r)^{k+1}} \geq (1-2\beta)+(1-2\beta)/2^{k+1}, 
\end{equation}
which is bigger than $(1+\beta)$ when $\beta$ is very small.

Now, since $\mathcal{H}^{k+1}(B)<\infty$, by enlarging an efficient $\delta$-cover to become one with cubes (gaining a multiplicative factor) we see that there is a constant $A$ (independent of $\beta,l$) such that for every $\delta>0$ there is a cover $\{C_i\}$ of $B^{l}_{\beta}$ by cubes such that $\sum diam(C_i)^{k+1} <A$ and $diam(C_i)<\delta$.  Taking $\delta<1/(4l)$ and such a good $\delta$-cover $\{C_i\}$, looking at $C_i$ in the parabolic metric, we see that  $diam_{par}(C_i)\leq \sqrt{diam(C_i)}$.  Now, slice each $C_i$ to rectangular boxes $\{D_i^j\}$ of time-like sides of length $diam(C_i)^2$ and the initial space-like sides. This way, the parabolic diameter of the boxes will be smaller than the Euclidean diameter of the original cube. We will need $\frac{1}{diam(C_i)}$ such rectangles and as in Lemma \ref{absolute_continuity} we see that
\begin{equation}
\sum_{j=1}^{[\frac{1}{diam(C_i)}]}diam_{par}(D_i^j)^{k+2}\leq[\frac{1}{diam(C_i)}] diam(C_i)^{k+2}=diam(C_i)^{k+1}
\end{equation}
and so $\mathcal{H}^{k+2}_{par}(B^{l}_{\beta}) < A$. In fact, our situation is  much better! Indeed, in light of (\ref{fewer}), only $6\beta[\frac{1}{diam(C_i)}]$ out of the $[\frac{1}{diam(C_i)}]$  rectangles can contribute to covering $B^{l}_{\beta}$. For if $(t,x)\in C_i$ we have $C_i \subseteq I^{n+1}((t,x),diam(C_i))$  and so by (\ref{fewer})
\begin{equation}
C_i \cap B^{l}_{\beta} \subseteq [t-3\beta diam(C_i),t+3\beta diam(C_i)]\times \mathbb{R}^{n}. 
\end{equation}
This gives
\begin{equation}
\mathcal{H}^{k+2}_{par}(B^{l}_{\beta})<6\beta A,
\end{equation} 
and so
\begin{equation}
\mathcal{H}^{k+2}_{par}\left(\bigcup_{l=1}^{\infty}B^{l}_{\beta}\right) \leq 6\beta A.
\end{equation}
As $\mathcal{H}^{k+1}\left(B-\bigcup_l B^{l}_{\beta}\right)=0$ we get
\begin{equation}
\mathcal{H}^{k+2}_{par}(B) \leq 6\beta A
\end{equation}
by the first part of Lemma \ref{absolute_continuity}. By the arbitrariness of $\beta$ we are done. 
\end{proof}

Computing the right hand side of the parabolic co-area formula in the infinitesimally spatial case is easier.

\begin{lemma}
For $B$  infinitesimally spatial we have
\begin{equation}
\int_{\mathbb{R}^{1,0}} \mathcal{H}^{k}(B_{t})d\mathcal{H}^{2}_{par}(t)=0.
\end{equation}
\end{lemma}
\begin{proof} This follows directly from the Euclidean co-area formula, as it implies that a.e. level set has $\mathcal{H}^{k}(B_t)=0$. 
\end{proof}

\subsection{The Time Advancing Part}\label{time_advancing}    

We first make several definitions.

\begin{definition}\label{time_advance_def}
A $k+1$ rectifiable set $\mathcal{M}$ is called \textbf{time-advancing} if for $\mathcal{H}^{k+1}$ a.e. $p\in \mathcal{M}$ we have that $\partial_t$ is not perpendicular to $T_p \mathcal{M}$ . 
\end{definition}

For technical reasons, it will be easier to work with the definitions below.

\begin{definition}
A Lipschitz (w.r.t the standard Euclidean metric) map $F:\mathbb{R}^{1,k} \supseteq A \rightarrow \mathbb{R}^{1,n}$ will be called \textbf{vertical} if  $\pi_{t}(F(t,x))=t$ for every $x\in \mathbb{R}^{k},t\in \mathbb{R}$. Here $\pi_{t}$ is the projection to the time  factor.
\end{definition}

\begin{definition}
A vertical map $F:\mathbb{R}^{1,k} \supseteq A \rightarrow \mathbb{R}^{1,n}$ will be called $\mathbf{(M,m)}$ \textbf{Lipschitz} if it is $M$ Lipschitz in the Euclidean sense, and if its restriction to every time slice is $m$ Lipschitz.
\end{definition}

\begin{definition}\label{ver_rect_def}
A set $\mathcal{M} \subseteq \mathbb{R}^{1,n}$ is said to be $(1,k)$ \textbf{vertically rectifiable} if one can write $\mathcal{M}=\mathcal{M}_{0}\cup \bigcup_{i\geq 1} \mathcal{M}_{i}$ where $\mathcal{H}^{k+1}(\mathcal{M}_{0})=0$ and where $\mathcal{M}_i=F_{i}(A_{i})$ , $A_{i}\subseteq \mathbb{R}^{1,k}$ are measurable and $F_{i}$ are $(M_i,m_i)$ Lipschitz.  
\end{definition}

The following lemma shows the equivalence between the geometric definition, Definition \ref{time_advance_def}, and the technical definition, Definition \ref{ver_rect_def}.

\begin{lemma}
Let $\mathcal{M}$ be a $k+1$ rectifiable in $\mathbb{R}^{1,n}$. Then $\mathcal{M}$ is time advancing iff it is vertically rectifiable.
\end{lemma}
\begin{proof} Assume $\mathcal{M}$ is time advancing. By rectifiability, write $\mathcal{M}=\mathcal{M}_0 \cup \bigcup_{i=1}^{\infty}\mathcal{M}_i$ where $\mathcal{H}^{k+1}(\mathcal{M}_0)=0$  and $\mathcal{M}_i \subseteq \mathcal{N}_i$ where $\mathcal{N}_i$ is an embedded $C^1$ submanifold in $\mathbb{R}^{1,n}$. We can therefore work on each $\mathcal{M}_i$ separately. Given $p \in \mathcal{M}_i$ with $\partial_t$ not perpendicular to $T_p\mathcal{M}_i$, this non-perpendicularity will also hold in an arbitrarily small ball around it in $\mathcal{N}_i$. In a yet smaller ball, we will be able to use the inverse function theorem with the first co-ordinate being $t$. Restricting it a little further will give an $(M,m)$ Lipschitz map. By Vitali covering we can get such a cover of the set and by the disjointedness of small balls, there are only countably many elements in that cover. The other implication is clear (and less important). 
\end{proof}

\vspace{5 mm}

At a stark contrast to the full dimensional case (see Lemma \ref{top_dim_comp}), the lower dimensional parabolic Hausdorff measures are far from the Euclidean ones, as the following example indicates.
\begin{example}\label{1d_example}
Let $F:\mathbb{R}^{1,0} \supset [a,b] \rightarrow \mathbb{R}^{1,1}$ be vertical, Lipschitz and increasing. Then $\mathcal{H}^{2}_{par}(F([a,b]))=\frac{\alpha(2)}{\alpha(1)}(b-a)$.
\end{example}
\begin{proof}  Let $M$ be the Lipschitz constant of $F$. For $\delta_{0}$ sufficiently small we have for every $\delta<\delta_{0}$ $M\delta^{2} < \delta$.  Thus, for every parabolic $\delta$ cover of $[a,b]\subseteq \mathbb{R}^{1,0}$  by $C_{i}$ and for every $t,s\in C_{i}$ we have $|t-s|<\delta^{2}$ so $|F(t)-F(s)|<\delta$. Thus, $F(C_i)$ is a $\delta$ cover of $F([a,b])$.  Similarly  we see that $diam_{par}(F(C_{i}))=diam_{par}(C_{i})$. Thus
\begin{equation}
\mathcal{H}^{2}_{par}(F([a,b]))\leq \mathcal{H}^{2}_{par}(([a,b]))=\frac{\alpha(2)}{2\alpha(1)}\mathcal{H}^{1}([a,b]). 
\end{equation}
The other direction is trivial. 
\end{proof}

\vspace{5 mm}

The main difference between the parabolic and Euclidean Hausdorff measures is captured by the following volume dilation estimate.

\begin{lemma}[Basic volume estimate]
Let $F:\mathbb{R}^{1,k} \supseteq A \rightarrow \mathbb{R}^{1,n}$ be an $(M,m)$ Lipschitz map. Then
\begin{equation}\label{basic_vol_est}
\mathcal{H}^{k+2}_{par}(F(A)) \leq \mathrm{max}\{m^{k},m^{k+2}\}\mathcal{H}^{k+2}_{par}(A). 
\end{equation}
\end{lemma}

\begin{proof} The proof is divided into four steps. In the first we consider what happens when $\mathcal{H}^{k+2}_{par}(A)=0$, in the second and third we derive the weaker inequality
\begin{equation}\label{weak_ineq}
\mathcal{H}^{k+2}_{par}(F(A)) \leq \mathrm{max}\{1,m^{k+2}\}\mathcal{H}^{k+2}_{par}(A),
\end{equation} 
and in the fourth we prove the strong inequality. 

\noindent\textbf{Step 1:} If $\mathcal{H}^{k+2}_{par}(A) =0$ then $\mathcal{H}^{k+2}_{par}(F(A)) =0$: Let $\{C_{i}\}$ be a parabolic $\delta$ cover of $A$  then $diam_{par}(F(C_{i}))\leq M\cdot diam_{par}(C_{i})$ from which it is clear.

\noindent\textbf{Step 2:} (\ref{weak_ineq}) holds if $A$ is a box, i.e. a set of the form $I \times B$  for $I\subseteq \mathbb{R}$ and $B\subseteq \mathbb{R}^{k}$: Take $\delta>0$ and let $\{C_i\}$ be a parabolic $\delta$ cover of $A$. Then $F(C_{i})$ is a cover of $F(A)$ and for every $(t,x),(s,y)\in A$ we have 
\begin{equation}
d_{par}(F(t,x),F(s,y))\leq \mathrm{max}\{\sqrt{|t-s|}, m|x-y|+M|t-s|\},  
\end{equation}
which implies
 \begin{equation}
\begin{aligned}
&diam_{par}(F(C_{i})) \leq \\
&\mathrm{max}\{diam_{par}(C_{i}),m\cdot diam_{par}(C_{i})+M\cdot diam_{par}(C_{i})^{2}\}\leq \\
&diam_{par}(C_{i}) \cdot \mathrm{max}\{1,m+M\delta\}.
\end{aligned}
\end{equation} 
Thus, assuming $\delta<1$ we obtain

\begin{equation}
\mathcal{H}^{k+2}_{par,\mathrm{max}\{1,m+M\}\delta}(F(A)) \leq \mathrm{max}\{1,m+M\delta\}^{k+2}\mathcal{H}^{k+2}_{par,\delta}(A)
\end{equation} 
and the desired result is obtained by taking $\delta \rightarrow 0$.

\noindent\textbf{Step 3:} In the general case, write $A=\bigcup A_{i}\cup B$ where
\begin{equation}
A_{i}=\{x\in A \mbox{ s.t } \Theta^{k+1}_{Euc}(x,A,r) \geq \frac{99}{100}  \mbox{ for every } 0<r<\frac{1}{i}\}   
\end{equation}
and
\begin{equation}
\Theta^{k+1}_{Euc}(x,A,r)=\frac{\mathcal{H}^{k+1}(A\cap B(x,r))}{\omega_{k+1}r^{k+1}}.
\end{equation}
 Then $\mathcal{H}^{k+2}_{par}(B)=\mathcal{L}^{k+1}(B)=0$ and $A_{i}\nearrow A-B$. Thus $\mathcal{H}^{k+2}_{par}(F(B))=0$,  $F(A)=\bigcup F(A_{i}) \cup F(B)$ and $F(A_{i}) \nearrow F(A)$ up to measure $0$. Thus, it suffices to show the desired weak inequality (\ref{weak_ineq}) for $A_{i}$.
Take $0<\delta<\frac{1}{i}$ and let $\{C_j\}$ be a parabolic $\delta$ cover of $A_{i}$ and assume further that $C_j \subseteq A_i$. Then $\{F(C_{j})\}$ is a cover of $F(A_{i})$ and for every $(t,x),(s,y)\in C_{j}$ (assume w.l.o.g. $s\leq t$) we have $|t-s|\leq diam_{par}(C_{j})^{2}$. Since both $(t,x)$ and $(s,y)$ are points of density at scale $\frac{1}{i}$ there will be some $s\leq r \leq t$ and points $x_{r},y_{r}\in \mathbb{R}^{k}$ with $|x_r-x|\leq diam_{par}(C_{j})^{2},|y-y_r| \leq diam_{par}(C_{j})^{2}$ and such that $(r,x_r),(r,y_r)\in A$. Thus, by the triangle inequality we get
\begin{equation}
\begin{aligned}
&d(\pi_{x}(F(t,x)),\pi_{x}(F_{2}(s,y)))\leq \\
&M\sqrt{|x-x_r|^{2}+|t-r|^{2}}+m|x_r-y_r|+M\sqrt{|y-y_r|^{2}+|s-r|^{2}} \leq \\
&4M\cdot diam_{par}(C_{i})^{2} +m(|x-y|+2diam_{par}(C_{i})^{2}) 
\end{aligned}
\end{equation} 
and the proof continues as in step 2.

\noindent\textbf{Step 4:} For the improved estimate (\ref{basic_vol_est}) note first that  we may assume $m<1$ or else it is equivalent to (\ref{weak_ineq}). Turning $C_j$ into a product set does not increase the parabolic diameter (because of the ``max''). Note further that if $diam_{par}(C_j)>max_{(t,x),(s,y)\in C_j}|x-y|$, it will be worthwhile to split $C_j$ into smaller product sets with the same space-like factor. Thus, in the product case (step 2) we can assume $C_j=[s,s+a^2]\times B_j$ where $diam(B)=b$ and $a \leq b$ . But then, we can split $[s,s+a^2]$ into $\frac{1}{m^2}$ intervals $I_{j,k}$, each of which of  length $m^2a^2$ and consider the cover $F(I_{j,k}\times B)$ of $F(C_j)$. Note that
\begin{equation}
diam_{par}(F(I_{j,k}\times B_j))\leq max\{ma,mb+Mm^2a^2\}=mb+Mm^2a^2\leq m\cdot diam_{par}(C_j)(1+M^2m\delta) 
\end{equation}      
Keeping in mind that we obtained $\frac{1}{m^2}$ such split boxes, this gives the desired result. In the general case, we argue as in step 3. 
\end{proof}

\begin{remark}
Note that in the $(M,m)$ Lipschitz setting, there is no effective extension theorem, in contrast to the Euclidean Lipschitz case, in which Kirszbraun's extension theorem  (see \cite[ Sec. 2.10.43]{Fed}) allows one to assume that the map is defined on the entire space (with the same Lipschitz constant). Thus the general assertion did not follow trivially from the one on boxes, and the third step was indeed needed. 
\end{remark}

\vspace{5 mm}

Motivated by the above, we make the following definition.
\begin{definition}
Suppose $F:\mathbb{R}^{1,k} \supset A \rightarrow \mathbb{R}^{1,n}$ is $(M,m)$  Lipschitz. The \textbf{horizontal differential} of $F$ at $(t_{0},x_{0})\in A$: $D^{h}F|_{(t_{0},x_{0})}$  is the differential of the map $F_{2}(t_{0},-):A\cap\{t=t_{0}\} \rightarrow \mathbb{R}^{n}$. The \textbf{horizontal Jacobian} $J^{h}F|_{(t_{0},x_{0})}$ is the Jacobian of that map. 
\end{definition}

\begin{remark}
Note that the above is well defined a.e. Indeed, by Fubini $A\cap\{t=t_{0}\}$ is measurable for almost every $t_{0}$ and we can Lipschitz extend in every such level set. The resulting differential is independent of the extension at points of density.
\end{remark}

\begin{definition}\label{vert_lin}
An $(n+1)\times (k+1)$ matrix $B$ is called \textbf{vertically linear} if it is of the form
\begin{equation}
B=\left(\begin{array}{cc}
1 & 0\\
v & A
\end{array}\right)
\end{equation}
for $v \in \mathbb{R}^{n}$ and $A$ an $n\times k$ matrix.
\end{definition}
\begin{remark}
The differential of an $(M,m)$ Lipschitz map is vertically linear.
\end{remark}

The following three auxiliary lemmas concerning the parabolic Hausdorff measure have their direct Euclidean analogues (see \cite[Sec 3.3.1 Lemmas 1-3]{EG}) with almost identical proofs. We will shortly remark about the (essentially cosmetic) differences.

\begin{lemma}
Suppose $F:\mathbb{R}^{1.k} \rightarrow \mathbb{R}^{1,n}$ is vertically linear,  then
\begin{equation}
\mathcal{H}^{k+2}_{par}(F(C))=(J^{h}F)\mathcal{H}^{k+2}_{par}(C).
\end{equation}
\end{lemma} 
\proofsketch  Writing  $A=O \circ S$ for $S:\mathbb{R}^{k}\rightarrow \mathbb{R}^{k}$ symmetric and $O:\mathbb{R}^{k}\rightarrow \mathbb{R}^n$ orthogonal we see that
\begin{equation}
F=\left(\begin{array}{cc}
1 & 0\\
v & Id
\end{array}\right)\left(\begin{array}{cc}
1 & 0\\
0 & O
\end{array}\right)\left(\begin{array}{cc}
1 & 0\\
0 & S
\end{array}\right)=\tilde{N}\tilde{O}\tilde{S}
\end{equation}

As both $\tilde{N}$ and its inverse are $(|v|+1,1)$ Lipschitz, $\tilde{N}^{-1}$ will preserve $\mathcal{H}^{k+2}_{par}$ (the proof of (\ref{basic_vol_est}) will work the same for $\tilde{O}\circ \tilde{S}(C)$ as it is Euclidean). Thus 
\begin{equation}
\mathcal{H}^{k+2}_{par}(F(C))=\mathcal{H}^{k+2}_{par}(\tilde{O}^{*}\tilde{N}^{-1}F(C))=\mathcal{H}^{k+2}_{par}(\tilde{S}(C)) 
\end{equation}
 so  by Lemma \ref{top_dim_comp} we are back in the Euclidean case. $\Box$

\begin{lemma}
Suppose $k\geq 1$ , $F:\mathbb{R}^{1.k} \supset A \rightarrow \mathbb{R}^{1,n}$ is $(M,m)$  Lipschitz for $A$ $\mathcal{H}^{k+2}_{par}$ measurable. Then:
\begin{enumerate}
\item $F(A)$ is $\mathcal{H}^{k+2}_{par}$ measurable.
\item The mapping $y \mapsto \mathcal{H}^{0}(A\cap F^{-1}\{y\})$ is $\mathcal{H}^{k+2}_{par}$ measurable on $\mathbb{R}^{1,n}$.
\item $\int_{\mathbb{R}^{1,n}}\mathcal{H}^{0}(A\cap F^{-1}\{y\})d\mathcal{H}^{k+2}_{par} \leq \max\{m^{k},m^{k+2}\}\mathcal{H}^{k+2}_{par}(A)$ .
\end{enumerate}
\end{lemma} 
\proofsketch The only difference here is that the standard Euclidean estimate  
\begin{equation}
\mathcal{H}^{n}(F(A)) \leq \mathrm{Lip}(F)^{n}\mathcal{H}^{n}(A)
\end{equation}
 is replaced by the corresponding parabolic estimate for $(M,m)$ Lipschitz functions (\ref{basic_vol_est}). $\Box$

We will often confuse a linear map $T:\mathbb{R}^k\rightarrow \mathbb{R}^n$ with the corresponding vertical map from $\mathbb{R}^{1,k}$ to $\mathbb{R}^{1,n}$ that splits time and space (i.e. $v=0$ in Definition \ref{vert_lin}).  
\begin{lemma}\label{partition_lemma}
Let $F:\mathbb{R}^{1.k} \supset A \rightarrow \mathbb{R}^{1,n}$ be an $(M,m)$ Lipschitz map , let $\alpha>1$ and let $B=\{x\in A \mbox{ s.t. } D^{h}F \mbox{ exists and } J^{h}F>0\}$. Then there is a countable collection of Borel subsets $\{E_{j}\}$ of $\mathbb{R}^{1,k}$ such that:
\begin{enumerate}
\item $B=\bigcup_{j=1}^{\infty} E_{j}$.
\item $F|_{E_{j}}$ is one to one.
\item for each $j$ there is a  symmetric automorphism  $T_{j}:\mathbb{R}^{k} \rightarrow \mathbb{R}^{k}$ such that (identifying it with the corresponding vertical map from $\mathbb{R}^{1,k}$ to $\mathbb{R}^{1,k}$):
\begin{enumerate}
\item $F|_{E_{j}}\circ T_{j}^{-1}$ is $(M_j,\alpha)$ Lipschitz. 
\item $T_{j}\circ (F|_{E_{j}})^{-1} $ is $(M_j,\alpha)$ Lipschitz.
\item $ \alpha^{-k}|\det T_{j}| \leq J^{h}F|_{E_{j}} \leq  \alpha^{k}|\det T_{j}|$.
\end{enumerate}
\end{enumerate}
\end{lemma}
\proofsketch This is also similar to the corresponding lemma \cite[3.3.1.3]{EG}. This time, fixing $\epsilon>0$ we let $C$ be a countable dense subset of $B$, $\mathbf{S}$ be a countable dense subset of the symmetric automorphisms of $\mathbb{R}^{k}$ and $\mathbf{W}$ be a countable dense subset of the vectors in $\mathbb{R}^{1,n}$ with first coordinate $1$. Then for $c\in C,T\in\mathbf{S},w\in \mathbf{W}$ and $i=1,2,3,\ldots$ we define $E(c,T,w,i)$ to be the set of all $b\in B\cap B(c,1/i)$ satisfying  
\begin{equation}
\left(\alpha^{-1}+\epsilon\right)|Tv| \leq |DG_w(b)v| \leq \left(\alpha-\epsilon\right)|Tv|
\end{equation}
for all $v\in \mathbb{R}^{1,k}$  and 
\begin{equation}
|G_w(a)-G_w(b)-DG_w(b)(a-b)|\leq \epsilon|T(a-b)|
\end{equation}
for all $a\in B\cap B(b,2/i)$. Here
\begin{equation}
G_w(y)=F(y)-\left<y,\partial_t\right>w.
\end{equation}
Then for $b\in E(c,T,w,i)$ we have 
\begin{equation}\label{comp_dif}
J^{h}F=J^{h}G_{w}
\end{equation}
 and just like in \cite{EG} we obtain
\begin{equation}
\left(\alpha^{-1}+\epsilon\right)^k\det(T) \leq J^hG_w(b) \leq \left(\alpha-\epsilon\right)^k\det(T)
\end{equation}
Now, choose any $b\in B$ and write $DF(b)=O\circ S+ \left<b,\partial_t\right>u$ (confusing $O,S$ with the corresponding vertical maps) and choose $T\in \mathbf{S}$ with $\mathrm{Lip}(T\circ S^{-1}) \leq \left(\alpha^{-1}+3\epsilon/2\right)^{-1}$ and $\mathrm{Lip}(S\circ T^{-1}) \leq \alpha-3\epsilon/2$ and $w\in \mathbf{W}$ with $|w-u|\leq \epsilon ||T||/2$ and select $i\in \{1,2,\ldots\}$ and $c\in C$ such that $|b-c|<1/i$ and
\begin{equation}
|F(a)-F(b)-DF(b)(a-b)| \leq \frac{\epsilon}{\mathrm{Lip}(T^{-1})}|a-b|
\end{equation} 
for all $a\in B\cap B(b,2/i)$. Then $b\in E(c,T,w,i)$. Renaming the sets $E(c,T,w,i)$ - $E_j$  will yield, just like in \cite{EG} a partition $\{E_{j}\}$ of $B$  with $\mathrm{Lip}(G_{w_j}|_{E_j} \circ T_j^{-1})<\alpha$ and $\mathrm{Lip}(T_j \circ \left(G_{w_j}|_{E_j}\right)^{-1} )<\alpha$  with the desired property. Translating $G_{w_j}$ back to $F$ will therefore give corresponding $(|w_j|+1,\alpha)$ Lipschitz maps and by (\ref{comp_dif}) we are done $\Box$ .

\vspace{5 mm}

We now come to the actual parabolic area formula for $(M,m)$ Lipschitz maps. Its proof is (again) identical to the one of the usual area formula (see \cite[Section 3.3.2]{EG}), with the above lemmas replacing the Euclidean ones and by using the parabolic $(M,m)$ Lipschitz estimate (\ref{basic_vol_est}). 

\begin{theorem}[Parabolic area formula]\label{par_area}
Let $F:\mathbb{R}^{1.k} \supset A \rightarrow \mathbb{R}^{1,n}$ be $(M,m)$ Lipschitz then
\begin{equation}
\int_{A}J^{h}Fd\mathcal{H}^{k+2}_{par}= \int _{\mathbb{R}^{1,n}} \mathcal{H}^{0}(A\cap F^{-1}(y))d\mathcal{H}^{k+2}_{par}(y).
\end{equation}
Moreover, if $g:\mathbb{R}^{1,n}\rightarrow \mathbb{R}$ is measurable
\begin{equation}
\int_{A}(J^{h}F)(g\circ F)d\mathcal{H}^{k+2}_{par}= \int _{\mathbb{R}^{1,n}} \mathcal{H}^{0}(A\cap F^{-1}(y))g(y)d\mathcal{H}^{k+2}_{par}(y).
\end{equation}
\end{theorem}
\proofsketch We use Lemma \ref{partition_lemma} instead of the usual Euclidean partition lemma. Then, in the original proof, it is crucial to obtain that $\alpha$ Lipschitz maps do not increase volume by much. We have the corresponding result using (\ref{basic_vol_est}) controlling the horizontal Lipschitz-constant. The full Lipschitz constant is of no interest, as it is absent from the estimate. $\Box$   

\subsection{Parabolic Co-Area Formula}\label{Fubini_conclusion}

\begin{proof}[Proof of Theorem D]  By Section \ref{horizontal_set} we know that the contribution of the infinitesimally spatial part of $\mathcal{M}$ to both sides is zero. We can thus suppose that  $\mathcal{M}$ is time advancing or equivalently, vertically rectifiable and in fact, that $\mathcal{M}=F(A)$ for $A \subseteq \mathbb{R}^{1,k}$ and $F:A\rightarrow \mathbb{R}^{1,n}$  $(M,m)$ Lipschitz and one to one. But in this case, by Theorem \ref{par_area}
\begin{equation}
\int _{\mathcal{M}} g d\mathcal{H}^{k+2}_{par}=\int_{A}(J^{h}F)(g\circ F)d\mathcal{H}^{k+2}_{par}=
\end{equation}
by Lemma \ref{top_dim_comp}
\begin{equation}
\begin{aligned}
&c_1\int_{\mathbb{R}^{1,0}}\left(\int_{\mathcal{A}_{t}}(J^{h}F)(g\circ F)d\mathcal{H}^{k}\right)d\mathcal{H}^{2}_{par}(t)=\\
&c_1\int_{\mathbb{R}^{1,0}}\left(\int_{\mathcal{M}_{t}}gd\mathcal{H}^{k}\right)d\mathcal{H}^{2}_{par}(t)
\end{aligned}
\end{equation}
where the last equality is by the Euclidean area formula. 
\end{proof}

\bibliographystyle{alpha}
\bibliography{isoBib}

\vspace{10mm}
{\sc Or Hershkovits, Courant Institute of Mathematical Sciences, New York University, 251 Mercer Street, New York, NY 10012, USA}\\

\emph{E-mail:}  or.hershkovits@cims.nyu.edu

\end{document}